\documentclass[12pt]{amsart}
\usepackage{amsfonts,amssymb,amscd,amstext,mathrsfs}
\usepackage[utf8]{inputenc}
\usepackage{hyperref}
\usepackage{verbatim}

\usepackage{graphics}
\usepackage{graphicx} 

\usepackage{times}
\usepackage{enumerate}
\usepackage[up,bf]{caption}
\usepackage{color}
\input xy
\xyoption{all}

\usepackage[color=blue!20]{todonotes}

\newtheorem{theorem}{Theorem}[section]

\theoremstyle{definition}

\newtheorem{remark}[theorem]{Remark}

\numberwithin{equation}{section}
\numberwithin{figure}{section}

\newcommand\Oscr{\mathscr{O}}

\newcommand\B{\mathbb{B}}
\newcommand\C{\mathbb{C}}

\newcommand\R{\mathbb{R}}

\newcommand\igot{\mathfrak{i}}

\renewcommand\igot{\mathfrak{i}}

\renewcommand\imath{\igot}

\newcommand\di{\partial}

\renewcommand\div{\mathrm{div}}

\newcommand\Aut{\mathrm{Aut}}

\newcommand\Id{\mathrm{Id}}

\def\Ell1{\mathrm{Ell_1}}
\def\DEll1{\mathrm{DEll_1}}

\numberwithin{equation}{section}

\makeatletter
\@namedef{subjclassname@2020}{\textup{2020} Mathematics Subject Classification}
\makeatother

\begin{document}
\title{Approximation of biholomorphic maps between Runge domains 
by holomorphic automorphisms
}
\author{Franc Forstneri\v{c}}

\address{Franc Forstneri\v c, Faculty of Mathematics and Physics, University of Ljubljana, and Institute of Mathematics, Physics, and Mechanics, Jadranska 19, 1000 Ljubljana, Slovenia}
\email{franc.forstneric@fmf.uni-lj.si}


\date{21 June 2025}


\begin{abstract}
We show that biholomorphic maps between certain pairs
of Runge domains in the complex affine space $\C^n$, $n>1$, 
are limits of holomorphic automorphisms of $\C^n$. A similar result
holds for volume preserving maps and also 
in Stein manifolds with the density property. This generalises 
several results in the literature with considerably simpler proofs.
\end{abstract}

\maketitle


%
%
%
%
\section{Introduction}\label{sec:intro} 

A holomorphic vector field $V$ on the complex Euclidean space $\C^n$
is said to be complete if its flow $\phi_t(z)$, solving the initial
value problem 
\[
	\frac{d}{dt}\phi_t(z)=V(\phi_t(z)), \quad \phi_0(z)=z\in \C^n,
\]
exists for every $z\in \C^n$ and $t\in\R$. Such a vector field $V$ is also 
complete in complex time $t\in\C$ 
(see \cite[Corollary 2.2]{Forstneric1996MZ}), 
and $\{\phi_t\}_{t\in\C}$ is a complex 1-parameter subgroup 
of the holomorphic automorphism group $\Aut(\C^n)$ of $\C^n$. 
The same conclusion holds if $V$ is assumed to be 
complete in positive real time;
see Ahern, Flores, and Rosay \cite{AhernFloresRosay2000}.

Let $\B(0,\epsilon)$ denote the ball of radius $\epsilon$ around 
the origin $0\in\C^n$. We say that $0$ is a 
{\em globally attracting fixed point} of $V$ if $V(0)=0$ 
and the following two conditions hold:
\begin{enumerate}
\item $\lim_{t\to+\infty}\phi_t(z) = 0$ holds for all $z\in\C^n$. 
\item For every $\epsilon>0$ there exists a $\delta>0$ such that 
$\phi_t(z)\in \B(0,\epsilon)$ for every $z\in \B(0,\delta)$ and $t\ge 0$.
\end{enumerate}
A domain $\Omega\subset \C^n$ 
is said to be invariant under the positive time flow of $V$ 
if $\phi_t(z)\in \Omega$ for every $z\in \Omega$ and $t\ge 0$. 
Such a domain is sometimes called {\em spirallike} for $V$ 
(see \cite{GrahamAll2008}).
It was shown by Chatterjee and Gorai  
\cite[Theorem 1.1]{ChatterjeeGorai2025X} 
(see also El Kasimi \cite{ElKasimi1988} 
for starshaped domains and Hamada \cite[Theorem 3.1]{Hamada2015} 
for linear vector fields) that a spirallike domain 
$\Omega$ containing the origin is Runge in $\C^n$, that is, the restrictions
of holomorphic polynomials on $\C^n$ to $\Omega$ form
a dense subset of the space $\Oscr(\Omega)$ of holomorphic
functions on $\Omega$. 

In this note we prove the following result.

%
%
\begin{theorem}\label{th:Runge}
Assume that $V$ is a complete holomorphic vector field 
on $\C^n$, $n>1$, with a globally attracting 
fixed point $0\in \C^n$ and the domain $0\in \Omega\subset \C^n$
is invariant under the positive time flow of $V$.
Then, every biholomorphic map from $\Omega$ 
onto a Runge domain in $\C^n$ can be approximated uniformly 
on compacts in $\Omega$ by holomorphic automorphisms of $\C^n$.
\end{theorem}

By Anders\'en and Lempert
\cite{AndersenLempert1992}, it follows that a biholomorphic map 
$F:\Omega\to F(\Omega)\subset\C^n$ in Theorem \ref{th:Runge}
can be approximated uniformly on compacts by 
compositions of holomorphic shears and generalized shears.
We refer to \cite[Chapter 4]{Forstneric2017E} 
and \cite{ForstnericKutzschebauch2022} for surveys 
of this theory.

For a starshaped domain $\Omega\subset\C^n$, Theorem 
\ref{th:Runge} coincides with \cite[Theorem 2.1]{AndersenLempert1992}
due to Anders\'en and Lempert.
Theorem \ref{th:Runge} generalizes results of Hamada 
\cite[Theorem 4.2]{Hamada2015} (in which the vector field $V$ is linear) 
and Chatterjee and Gorai \cite[V5, Theorem 1.5]{ChatterjeeGorai2025X}. 
Their results give the same conclusion under additional conditions 
on the flow of $V$, and their proofs (especially the one in 
\cite{ChatterjeeGorai2025X}) are fairly involved. 
The papers \cite{Hamada2015,ChatterjeeGorai2025X} 
include applications to the theory of Loewner 
partial differential equation; see Arosio, Bracci and Wold 
\cite{ArosioBracciWold2013} for the latter.

Here we show that Theorem \ref{th:Runge} is an 
elementary corollary to \cite[Theorem 1.1]{ForstnericRosay1993} 
and no additional conditions on the vector field $V$ are necessary.

We wish to point out that very little seems to be known about 
globally attracting complete nonlinear holomorphic vector fields
on $\C^n$ for $n>1$. 
It was proved by Rebelo \cite{Rebelo1996} that a complete holomorphic 
vector field on $\C^2$ has a nonvanishing two-jet at each fixed point. 
It seems unknown whether such a vector field 
can have more than one fixed point.

\begin{proof}[Proof of Theorem \ref{th:Runge}] 
Let $V$ and $\Omega\subset \C^n$ be as in the theorem.
By \cite[Theorem 1.1]{ChatterjeeGorai2025X}, 
$\Omega$ is Runge in $\C^n$. We shall prove 
that every biholomorphic map $F:\Omega\to \Omega'$ 
onto a Runge domain $\Omega'\subset\C^n$ 
is a limit of holomorphic automorphisms
of $\C^n$, uniformly on compacts in $\Omega$. 

We may assume that $F(0)=0$ and the derivative
$DF(0)$ is the identity map. 
Thus, $F$ is a small perturbation of the identity near the origin.
In particular, choosing $\epsilon >0$ small enough, 
we have that $\B(0,\epsilon)\subset\Omega$ and 
the image $F(\B(0,\epsilon))$ is convex. Hence, the restricted map
$F:\B(0,\epsilon)\to F(\B(0,\epsilon))$ is a limit of  
holomorphic automorphisms by \cite[Theorem 2.1]{AndersenLempert1992}.
Fix such an $\epsilon$. Note that for each $t\ge 0$ the domain 
$\Omega_t:=\phi_t(\Omega)\subset \Omega$
is Runge in $\C^n$ (and hence in $\Omega$) since $\phi_t\in\Aut(\C^n)$. 

Assume first that $\overline \Omega$ is compact.  
Conditions (1) and (2) on the vector field 
$V$ imply that there is a $t_0>0$ such that 
$\phi_{t}(\overline\Omega)\subset \B(0,\epsilon)$ for all $t\ge t_0$.
Indeed, given a point $p\in \overline \Omega$, condition (1) 
gives a number $t(p)>0$ such that $\phi_{t(p)}(p)\in \B(0,\delta)$.
By continuity, there is a neighbourhood $U_p\subset \C^n$
of $p$ such that $\phi_{t(p)}(U_p)\subset \B(0,\delta)$.
This gives a finite open cover $U_1,\ldots,U_m$ of $\overline \Omega$
and numbers $t_1>0,\ldots,t_m>0$ such that 
\begin{equation}\label{eq:delta}
	\phi_{t_j}(U_j)\subset \B(0,\delta)\ \ \text{holds for $j=1,\ldots,m$}.
\end{equation}
Set $t_0=\max\{t_1,\ldots,t_m\}$. By property (2) of the flow 
and \eqref{eq:delta} we have that 
$\phi_t(\overline \Omega)\subset \B(0,\epsilon)$ for all $t\ge t_0$,
which proves the claim.

Recall that $\Omega'=F(\Omega)$. 
Let $\psi_t:\Omega'\to \Omega'$ for $t\ge 0$ be the unique 
holomorphic map which is $F$-conjugate to $\phi_t:\Omega\to\Omega$,
defined by the condition
\[
	F\circ \phi_t = \psi_t \circ F\ \ \text{for all $t\ge 0$}.
\]
Thus, $\psi_t$ maps $\Omega'$ biholomorphically onto the domain
$
	\Omega'_t:=\psi_t(\Omega') =F(\Omega_t)\subset\Omega'
$
for every $t\ge 0$, and $\psi_0$ is the identity on $\Omega'$. 
Since $\Omega_t$ is Runge in $\Omega$ for every $t\ge 0$ 
and the map $F:\Omega\to\Omega'$ is biholomorphic, we infer
that $\Omega'_t$ is Runge in $\Omega'$, 
and hence also in $\C^n$ (since $\Omega'$ is Runge in $\C^n$). 
Consider the family of maps 
\begin{equation}\label{eq:isotopy}
	F_t = F\circ \phi_t : 
	\Omega \stackrel{\cong}{\longrightarrow} \Omega'_t,
	\quad t\ge 0.
\end{equation}
This is an isotopy of biholomorphic maps from the Runge domain $\Omega$
onto the family of Runge domains $\Omega'_t\subset \C^n$, 
with $F_t$ depending smoothly on $t$. Since 
$\overline \Omega_{t_0}= \phi_{t_0}(\overline \Omega)\subset \B(0,\epsilon)$, 
the restricted map $F:\B(0,\epsilon)\to \C^n$ is a limit of 
automorphisms of $\C^n$ and $\phi_{t_0}\in\Aut(\C^n)$,
the map $F_{t_0}=F\circ\phi_{t_0}$ is 
also a limit of automorphisms of $\C^n$. By 
\cite[Theorem 1.1]{ForstnericRosay1993} 
it follows that every map $F_t$ in the isotopy \eqref{eq:isotopy}
is a limit of automorphisms of $\C^n$. In particular, this holds for
the map $F_0=F:\Omega\to\Omega'$.

This proves the theorem in the case when $\overline\Omega$
is compact. The general case follows by observing that 
$\Omega$ is exhausted by relatively compact 
domains $\Omega_0\Subset \Omega$ containing the origin
which are invariant under the positive time flow of $V$. 
To see this, choose an open relatively compact subset $W$ of $\Omega$ 
and set $\Omega_0=\bigcup_{t\ge 0} \phi_t(W)\subset \Omega$.
Obviously, $\Omega_0$ is open and positive time invariant.
Pick $\epsilon>0$ such that $\overline{\B(0,\epsilon)}\subset\Omega$.
We see as before that there is a number $t_0>0$ 
such that $\phi_t(\overline W)\subset \B(0,\epsilon)$ for all $t\ge t_0$.
It follows that 
\[
	\Omega_0 \subset 
	\bigcup_{0\le t\le t_0}\phi_t(\overline W) \cup 
	\overline{\B(0,\epsilon)}.
\]
Since the first set on the right hand side is compact
and contained in $\Omega$, we see that 
$\overline \Omega_0\subset\Omega$.
By \cite[Theorem 1.1]{ChatterjeeGorai2025X}, 
$\Omega_0$ is Runge in $\C^n$, and hence in $\Omega$.
If follows that $F(\Omega_0)=\Omega'_0$ is Runge 
in $\Omega'=F(\Omega)$, and hence also in $\C^n$ 
(since $\Omega'$ is Runge in $\C^n$). The above argument 
in the special case then shows that $F:\Omega_0\to  \Omega'_0$ 
is a limit of holomorphic automorphisms of $\C^n$ uniformly on
compacts in $\Omega$. 
By the construction, $\Omega_0$ can be chosen to contain 
any given compact subset of $\Omega$, which proves the theorem.
\end{proof}

%
%
The Runge domain $\Omega$ in Theorem \ref{th:Runge}
need not be pseudoconvex. 
Replacing $\C^n$ by a Stein manifold $X$ with the density property
(see Varolin \cite{Varolin1999,Varolin2000} and 
\cite[Section 4.10]{Forstneric2017E}) and assuming 
that $\Omega$ is a pseudoconvex Runge domain in $X$
which is positive time invariant for a complete 
holomorphic vector field $V$ on $X$ with a globally attracting 
fixed point in $\Omega$, the conclusion of Theorem \ref{th:Runge} 
still holds, with the same proof. The relevant version of 
the result on approximation of isotopies of biholomorphic maps
between pseudoconvex Runge domains in $X$ by holomorphic
automorphisms of $X$ is given by
\cite[Theorem 4.10.5]{Forstneric2017E}.
(A recent survey on Stein manifolds with the density property
can be found in \cite[Sect.\ 2]{ForstnericKutzschebauch2022}.) 
However, we do not know any example of a 
Stein manifold with the density property
and with a globally attracting complete holomorphic vector field,
other than the Euclidean spaces $\C^n$, $n>1$.

%
%
A version of Theorem \ref{th:Runge} also holds for biholomorphic
maps $F:\Omega\to\Omega'$ between certain Runge domains 
in $\C^n$ with coordinates $z_1,\ldots,z_n$ 
preserving the holomorphic volume form
\begin{equation}\label{eq:omega}
	\omega=d z_1\wedge \cdots\wedge d z_n,
\end{equation} 
in the sense that $F^*\omega=\omega$. 
Note that $F^*\omega = (JF)\, \omega$
where $JF$ denotes the complex Jacobian determinant of $F$.
Recall that the {\em divergence} of a holomorphic vector field         			
$V$ with respect to $\omega$ is the holomorphic function 
$\div_\omega V$ satisfying the equation
\begin{equation}\label{eq:divergence}
    L_V \omega= d (V\rfloor \omega) + V\rfloor d\omega
    = d (V\rfloor \omega) = \div_\omega V\cdotp \omega,
\end{equation}
where $L_V \omega$ is the Lie derivative of $\omega$
and $V\rfloor \omega$ is the inner product of $V$ and
$\omega$. The first equality is Cartan's formula
(see \cite[Theorem 6.4.8]{AbrahamMarsdenRatiu1988}), 
and we used that $d\omega=0$.
Let $\phi_t$ denote the flow of $V$. 
From \eqref{eq:divergence} we obtain Liouville's formula 
\begin{equation}\label{eq:Liouville}
    \frac{d}{dt}\, \phi_t^*\omega = \phi_t^* (L_V \omega) 
     =  \phi_t^* (\div_\omega V \cdotp \omega).
\end{equation}
Assume now that $\div_\omega V=c\in \C$ is constant. 
This holds in particular for every linear holomorphic vector field on $\C^n$ 
as is seen from the formula
\begin{equation}\label{eq:explicit}
    \div_\omega \biggl( \sum_{j=1}^n a_j(z) \frac{\di}{\di z_j}\biggr) 
    = \sum_{j=1}^n  \frac{\di a_j}{\di z_j}(z).
\end{equation}
In this case, \eqref{eq:Liouville} reads 
$\frac{d}{dt}\, \phi_t^*\omega=c \, \phi_t^*\omega$. Since 
$\phi_0=\Id$, it follows that 
\begin{equation}\label{eq:ec}
	\phi_t^*\omega=e^{ct} \omega\ \ \text{for all $t$}.
\end{equation}
In particular, if $V$ is globally contracting then $\Re c<0$. 
The case $c=0$ corresponds to $\omega$-preserving
vector fields whose flow maps have Jacobian $1$. 
The following result should be compared  
with \cite[V5, Theorem 1.10 (i)]{ChatterjeeGorai2025X}.
As before, $\omega$ is given by \eqref{eq:omega}.

%
%
\begin{theorem}\label{th:Runge-volume}
Let $V$ be a complete holomorphic vector field 
on $\C^n$, $n>1$, with a globally attracting fixed point $0\in \C^n$,  
whose divergence $\div_\omega V=c$ is constant. 
Assume that the domain $0\in \Omega\subset \C^n$
is pseudoconvex, invariant under the positive time flow 
$\{\phi_t\}_{t\ge 0}$ of $V$, it satisfies $H^{n-1}(\Omega,\C)=0$, and 
$\phi_{t_0}(\Omega)\Subset \Omega$ holds for some $t_0>0$. 
Then, every volume preserving biholomorphic map 
of $\Omega$ onto a Runge domain $\Omega'\subset \C^n$ 
can be approximated uniformly on compacts in $\Omega$ 
by volume preserving automorphisms of $\C^n$.
\end{theorem}

By Anders\'en \cite{Andersen1990}, 
every volume preserving holomorphic automorphism of $\C^n$
is a locally uniform limit of compositions of shears.

\begin{proof}
Since $V$ is globally attracting, we have that $\Re c<0$.
Let $W= \frac{-c}{n} \sum_{j=1}^n z_j \frac{\di}{\di z_j}$.
From \eqref{eq:explicit} we see that $\div_\omega W =-c$. 
The flow $\psi_t$ of $W$ is complete on $\C^n$ and satisfies 
\begin{equation} \label{eq:e-c}
	\psi_t^*\omega=e^{-ct}\omega\ \ \text{for all $t\in\C$}.  
\end{equation}
(Compare with \eqref{eq:ec}.) 
Consider the family of injective holomorphic maps 
\[
	F_t := \psi_t \circ F \circ \phi_t:
	\Omega\to F_t(\Omega)\subset\C^n,
	\quad t\ge 0.
\]
Note that $F_0=F:\Omega\to\Omega'$. Since 
$JF=1$, it follows from \eqref{eq:ec} and \eqref{eq:e-c} 
that $JF_t=1$ for all $t\ge 0$. The conclusion now follows by 
the same argument as in the proof 
of Theorem \ref{th:Runge}, using the second part
of \cite[Theorem 1.1]{ForstnericRosay1993} on 
approximation of isotopies of volume preserving biholomorphic 
maps by volume preserving automorphisms of $\C^n$.
(See the Erratum to \cite{ForstnericRosay1993} 
concerning the condition $H^{n-1}(\Omega,\C)=0$.)
\end{proof}

%
%
\begin{remark}
(A) Theorem \ref{th:Runge-volume} can be generalized to 
Stein manifolds $(X,\omega)$ having the volume density property; 
see \cite{Varolin2000}, \cite{ForstnericKutzschebauch2022},
and \cite[Sect.\ 4.10]{Forstneric2017E} for this topic.

(B) Chatterjee and Gorai stated an analogue of   
\cite[V5, Theorem 1.5]{ChatterjeeGorai2025X} for holomorphic
vector fields on $\C^{2n}$ with coordinates 
$(z_1,\ldots,z_n,w_1,\ldots,w_n)$ preserving the 
holomorphic symplectic form $\omega=\sum_{j=1}^n dz_j\wedge dw_j$ 
\cite[V5, Theorem 1.10 (ii)]{ChatterjeeGorai2025X}. 
Note however that such a vector field also preserves
the volume form $\omega^n$, 
so it does not have any attracting fixed points.
\end{remark}

%
%
\smallskip
\noindent {\bf Acknowledgements.} 
Forstneri\v c is supported by the European Union 
(ERC Advanced grant HPDR, 101053085) 
and grants P1-0291 and N1-0237 from ARIS, Republic of Slovenia.




\end{document}